\documentclass[12pt, leqno]{amsart}
\usepackage{amsmath, amssymb, amsfonts, amsthm, amscd, mathrsfs, enumerate}
\usepackage{fullpage}
\usepackage{txfonts,graphics,float,cancel}
\usepackage[all]{xy}

\RequirePackage[dvipsnames,usenames]{color}
\usepackage{hyperref}

\usepackage{txfonts}

\DeclareMathOperator{\Exc}{Exc}
\DeclareMathOperator{\Supp}{Supp}
\DeclareMathOperator{\Lie}{Lie}
\DeclareMathOperator{\Spec}{Spec}

\newtheorem{theorem}{Theorem}[section]
\newtheorem{corollary}[theorem]{Corollary}
\newtheorem{lemma}[theorem]{Lemma}
\newtheorem{proposition}[theorem]{Proposition}
\newtheorem{claim}[theorem]{Claim}
\newtheorem*{mainthm}{Main Theorem}
\newtheorem*{maincor}{Main Corollary}

\theoremstyle{definition}
\newtheorem{definition}[theorem]{Definition}
\newtheorem{remark}[theorem]{Remark}
\newcommand{\cO}{\mathcal{O}}
\newcommand{\bQ}{\mathbb{Q}}
\newcommand{\tld}{\tilde}
\DeclareMathOperator{\Hom}{{Hom}}

\newcommand{\Ical}{\mathcal{I}}
 \newcommand{\Ocal}{\mathcal{O}}

\newcommand{\veps}{\varepsilon}
\newcommand{\End}{\operatorname{End}}
\newcommand{\bg}{\mathbb{G}}
\newcommand{\co}{\mathcal{O}}
\newcommand{\al}{\alpha}
\newcommand{\beqn}{\begin{equation}}
\newcommand{\eeqn}{\end{equation}}

\def\Q{\mathbb{Q}}
\usepackage[all]{xy}
\newcommand{\eqn}{\begin{eqnarray*}}
\newcommand{\eneqn}{\end{eqnarray*}}

\newcommand{\ext}{\mathscr{E}xt}
\renewcommand{\co}{\mathcal{O}}
\newcommand{\depth}{\operatorname{depth}}
\newcommand{\bc}{\mathbb{C}}
\begin{document}

\title{Richardson Varieties Have Kawamata Log Terminal Singularities}
\author{Shrawan Kumar and Karl Schwede}

\begin{abstract}
Let $X^v_w$ be a Richardson variety in the full flag variety $X$ associated to a
symmetrizable
Kac-Moody group $G$. Recall that $X^v_w$ is the intersection of the finite
dimensional Schubert variety $X_w$ with the finite
codimensional opposite Schubert variety $X^v$. We give an explicit $\bQ$-divisor
 $\Delta$ on $X^v_w$ and prove that the pair $(X^v_w, \Delta)$  has  Kawamata
 log
terminal singularities.
  In fact, $-K_{X^v_w} - \Delta$ is ample, which additionally proves
 that $(X^v_w, \Delta)$ is log Fano.

We first give a proof of our result in the finite case (i.e., in the case when $G$
is a finite dimensional semisimple group)
 by a careful analysis of an explicit  resolution of singularities of $X^v_w$
 (similar to the BSDH resolution of the Schubert varieties).
  In the general Kac-Moody case, in the absence of an explicit resolution
  of $X^v_w$ as above, we give a proof that relies on the
  Frobenius splitting methods. In particular, we use Mathieu's result
  asserting that the Richardson varieties are Frobenius split, and combine it
  with a result of N.~Hara and K.-I.~Watanabe relating Frobenius splittings
  with log canonical singularities.
\end{abstract}

\thanks{The first author was partially supported by the NSF grant DMS \#0901239.}
\thanks{The second author was partially supported by the NSF grant DMS \#1064485.}


\address{S.K.: Department of Mathematics, University of North Carolina,
Chapel Hill, NC 27599-3250, USA}
\email{shrawan$@$email.unc.edu}

\address{K.S.: Department of Mathematics, Penn State University,
University Park, PA 16802, USA}
\email{schwede$@$math.psu.edu}

\maketitle

\section{Introduction}
Let $G$ be any symmetrizable Kac-Moody group over $\mathbb{C}$ (or any algebraically
closed field of characteristic zero) with the standard Borel
subgroup $B$, the standard negative Borel subgroup $B^{-}$, the
maximal torus $T=B\cap B^{-}$ and the Weyl group $W$.  Let $X=G/B$ be the full
flag
variety. For any $w\in W$, we
have the Schubert variety
$$
X_{w}:=\overline{BwB/B}\subset G/B
$$
and the opposite Schubert variety
$$
X^{w}:=\overline{B^{-}wB/B}\subset G/B.
$$
For any $v\leq w$, consider the \emph{Richardson variety $X^{v}_{w}$} which is
 defined to be the intersection of a Schubert variety and an opposite Schubert
 variety.
$$
X^{v}_{w}:=X^{v}\cap X_{w}
$$
with the reduced subscheme structure.
In this paper, we prove the following theorem.
\begin{mainthm} [Theorem \ref{thm2.2}]
With the notation as above, for any $v\leq w\in W$, there exists an effective
 $\bQ$-divisor $\Delta$ on $X^v_w$ such that $(X^v_w, \Delta)$ has Kawamata log
 terminal (for short KLT) singularities.

 Furthermore, $-K_{X^v_w} - \Delta$ is ample which proves that $(X^v_w, \Delta)$
 is also log Fano.
\end{mainthm}
This divisor $\Delta$, described  in Section \ref{sec2},  is built out
of the boundary $\partial X^v_w$.
As an immediate corollary of this result and Kawamata-Viehweg vanishing, we obtain
the following cohomology vanishing (due to Brion--Lakshmibai in the finite case).

\begin{maincor} [Corollary \ref{coro2.3}]
For a dominant integral weight $\lambda$, and any $v\leq w$,
$$
H^{i}(X^{v}_{w},\mathcal{L}(\lambda)_{|X^{v}_{w}})=0,\quad\text{for
  all}\quad i>0.
$$
\end{maincor}

Note, KLT singularities are a refinement of rational singularities.  In particular,
every KLT singularity is also a rational singularity, but not conversely except
in the Gorenstein case.  We note that, in the finite case,  Richardson varieties
have rational singularities
\cite[Theorem 4.2.1]{BrionLecturesOnGeometryOfFlagVarieties}, even in
positive characteristics \cite[Appendix]{KnutsonLamSpeyerProjectionsOfRichardson}.
 Indeed the singularities of generalizations of Richardson varieties has been a
 topic of interest lately \cite{BilleyCoskunSingsOfRichardson},
 \cite{KnutsonLamSpeyerProjectionsOfRichardson} and \cite{KnutsonWooYongSingularitiesOfRichardson}.

On the other hand, KLT are the widest class of singularities for which the foundational theorems of
the minimal model program over $\mathbb{C}$ are known to hold \cite{KollarMori}.  It is well known that toric varieties are KLT \cite[Section 11.4]{CoxLittleSchenckToricVarieties} and more generally V.~Alexeev and M.~Brion proved that spherical varieties are KLT \cite{AlexeevBrionStableReductiveVarieties2}.
Recently, D.~Anderson and A.~Stapledon proved that the Schubert varieties $X_w$
are log Fano and thus also KLT \cite{AndersonStapledonSchubertVarietiesAreLogFano},
also see \cite{HsiaoMultiplierIdealsOnDeterminantalVarieties}.

The proof of  our main result in the finite case is much simpler than the general
 Kac-Moody case and is given in Section 4.
 In this case  we are able to directly prove that  $(X^v_w, \Delta)$ is KLT
 through an explicit resolution of singularities of $X^v_w$  due to M. Brion (similar to
 the Bott-Samelson-Demazure-Hansen desingularization of the
 Schubert varieties).

In the general symmetrizable Kac-Moody case,  we are not aware of an explicit
resolution of singularities of $X^v_w$ to proceed as above. In the general
case, we prove our main result
by reduction to characteristic $p > 0$.  In this case we use an unpublished result
of O.~Mathieu asserting that the Richardson varieties $X^v_w$ are Frobenius split
compatibly splitting  their boundary (cf. Proposition \ref{prop4.1}).  This
splitting together with results of N.~Hara and K.-I.~Watanabe relating Frobenius
splittings and log canonical singularities (cf. Theorem \ref{thm4.5}) allow us
to conclude that the pair  $(X^v_w, \Delta)$ as above is KLT.

%\todo{{\bf Karl:} Just to remind myself to add references to Billey and Coskun and also to Knutson, Lam, Speyer.}

\noindent
{\bf Acknowledgements:} We thank V.~Alexeev, D.~Anderson, M. Brion,  M. Kashiwara and O. Mathieu for some
helpful correspondences.  We also thank the referee for numerous helpful comments and suggestions.

\section{Preliminaries and definitions}\label{sec1}

We follow the notation from \cite[Notation 0.4]{KollarMori}.  We  fix $X$ to be a
normal variety over an algebraically closed field.

Suppose that $\pi : \tilde{X} \to X$ is a proper birational map with $\tilde X$
normal.  For any $\bQ$-divisor $\Delta = \sum_i d_i D_i$ on $X$, we let $\Delta'
 = \pi^{-1}_* \Delta = \sum_{i} d_{i}D'_{i}$ denote the \emph{strict transform} of
 $\Delta$ defined as the $\Q$-divisor on $\tilde{X}$,
where $D'_{i}$ is the prime divisor on $\tilde{X}$ which is
birational to $D_{i}$ under $\pi$.  We let $\Exc(\pi)$ of $\pi$ be the exceptional
 set of $\pi$; the closed subset of $\tilde{X}$ consisting of those $x\in \tilde{X}$
 where $\pi$ is not biregular at $x$.  We endow $\Exc(\pi)$ with the reduced (closed) subscheme structure.
An (integral) divisor $D=\sum n_{i}F_{i}$ is called a {\em canonical divisor}
$K_{X}$ of $X$ if the restriction $D^{o}$ of $D$, to the smooth
locus $X^{o}$ of $X$, represents the canonical line bundle
$\omega_{X^{o}}$ of $X^{o}$.

Assume now that $K_{X}+\Delta$ is $\mathbb{Q}$-Cartier, i.e., some
multiple $n(K_{X}+\Delta)$ (for $n\in \mathbb{N}$) is a Cartier
divisor.  We may choose $K_{\tld X}$ that agrees with $K_X$ wherever $\pi$ is
 an isomorphism and thus it follows that there exists a (unique)
$\mathbb{Q}$-divisor $E_{\pi}(\Delta)$ on $\tilde{X}$ supported in
$\Exc(\pi)$ such that
\begin{equation}
n(K_{\tilde{X}}+\Delta') =
\pi^{*}(n(K_{X}+\Delta))+nE_{\pi}(\Delta). \label{eq1}
\end{equation}

A $\mathbb{Q}$-divisor $D=\sum d_{i}D_{i}$ on a smooth variety
$\tilde{X}$ is called a {\em simple normal crossing divisor} if each
$D_{i}$ is smooth and they intersect transversally at each
intersection point (in particular, this means that locally analytically the $D_i$ can be thought of as coordinate hyperplanes).

%The support of a divisor $D=\sum d_{i}D_{i}$ is by definition the closed (reduced) subvariety
%$$ \Supp D=\bigcup_{d_{i}\neq 0}D_{i}.$$

Let $X$ be an irreducible variety and $D$ a $\mathbb{Q}$-divisor on
$X$. A {\em
  log resolution} of $(X,D)$ is a proper birational morphism
$\pi:\tilde{X}\to X$ such that $\tilde{X}$ is smooth, $\Exc(\pi)$ is a
divisor and $\Exc(\pi)\cup \pi^{-1}(\Supp D)$ is a simple normal
crossing divisor. Log resolutions exist for any $(X,D)$ in characteristic zero
 by \cite{HironakaResolution}.

Let $X$ be a proper scheme. Then a $\mathbb{Q}$-Cartier
$\mathbb{Q}$-divisor $D$ is called {\em nef} (resp., {\em big}) if
$D\cdot C\geq 0$, for every irreducible curve $C\subset X$ (resp.,
$ND$ is the sum of an ample and an effective divisor, for some $N\in
\mathbb{N}$) (cf. \cite[\S\S 0.4 and 2.5]{KollarMori}). Recall that an ample
Cartier divisor is nef and big.

\begin{definition}\label{defi1.2}
 Let $X$ be a normal irreducible variety over a field of characteristic zero and let
$\Delta=\sum_{i} d_{i}D_{i}$ be a $\mathbb{Q}$-divisor with
$d_{i}\in [0,1)$. The pair $(X,\Delta)$ is called {\em Kawamata
log terminal}
  (for short {\em KLT}) if the following two conditions are satisfied:
\begin{itemize}
\item[{\rm(a)}] $K_{X}+\Delta$ is a $\mathbb{Q}$-Cartier $\mathbb{Q}$-divisor, and

\item[{\rm(b)}] There exists a log resolution $\pi:\tilde{X}\to X$ of $(X, \Delta)$ such that
the $\mathbb{Q}$-divisor $E=E_{\pi}(\Delta)=\sum_{i}e_{i}E_{i}$, defined by
 \eqref{eq1},
 satisfies
\begin{equation}
-1<e_{i}\quad\text{for all $i$.}\label{eq2}
\end{equation}
\end{itemize}

By \cite[Remarks 7.25]{DebarreHigherDimensionalAG}, $(X,\Delta)$ satisfying (a) is
KLT if and only if for every
proper birational map $\pi':Y\to X$ with normal $Y$, the divisor
$E_{\pi'}(\Delta)=\sum_{j}f_{j}F_{j}$ satisfies
\eqref{eq2}, i.e., $-1<f_{j}$ , for all $j$.  In fact, one may use this condition as a definition of KLT singularities in characteristic $p > 0$ (where it is an open question whether or not log resolutions exist).
\end{definition}

For a normal irreducible variety $X$ of characteristic zero with a
$\mathbb{Q}$-divisor $\Delta = \sum_i d_i D_i$ with $d_i \in [0,1]$, the pair $(X,\Delta)$ is called {\em
log canonical} if it satisfies the above conditions (a) and (b) with
\eqref{eq2} replaced by
\begin{equation}
-1\leq e_{i}\quad\text{for all $i$.}\label{eq3}
\end{equation}

\begin{remark}  Let $X$ be a variety of characteristic zero.
It is worth remarking that if $(X, \Delta)$ is KLT,  then $X$ has rational
singularities \cite{ElkikRationalityOfCanonicalSings}, \cite[$\S$5.22]{KollarMori}.
 Conversely, if $K_X$ is Cartier and $X$ has rational singularities, then $(X, 0)$
 is KLT \cite[$\S$5.24]{KollarMori}.
\end{remark}

We conclude this section with one final definition.

\begin{definition}
If $X$ is projective, we say that a pair $(X, \Delta)$ is \emph{log Fano} if
$(X, \Delta)$ is KLT and $-K_X - \Delta$ is ample.
\end{definition}

\section{Statement of the main result and its consequences}\label{sec2}

\subsection{Notation}\label{sec2.1}

Let $G$ be any symmetrizable Kac-Moody group over a field of characteristic zero with the standard Borel
subgroup $B$, the standard negative Borel subgroup $B^{-}$, the
maximal torus $T=B\cap B^{-}$ and the Weyl group $W$ (cf.
\cite[Sections 6.1 and 6.2]{KumarKacMoodyGroupsBook}). Let $X=G/B$ be the full flag
variety, which is a projective ind-variety. For any $w\in W$, we
have the Schubert variety
$$
X_{w}:=\overline{BwB/B}\subset G/B
$$
and the opposite Schubert variety
$$
X^{w}:=\overline{B^{-}wB/B}\subset G/B.
$$

Then, $X_{w}$ is a (finite dimensional) irreducible projective
subvariety of $G/B$ and $X^{w}$ is a finite codimensional
irreducible projective ind-subvariety of $G/B$ (cf.
\cite[Section 7.1]{KumarKacMoodyGroupsBook}). For any integral weight
$\lambda$ (i.e., any character
$e^{\lambda}$ of $T$), we have a $G$-equivariant line bundle
$\mathcal{L}(\lambda)$ on $X$ associated to the character
$e^{-\lambda}$ of $T$ (cf. \cite[Section 7.2]{KumarKacMoodyGroupsBook} for a precise
definition of $\mathcal{L}(\lambda)$ in the general Kac-Moody case). In the finite
case, recall that $\mathcal{L}(\lambda)$ is the line bundle associated to the
principal $B$-bundle $G \to G/B$ via the character $e^{-\lambda}$ of $B$ (any
character
of $T$ uniquely extends to a character of $B$), i.e.,
\[\mathcal{L}(\lambda)= G\times^B\mathbb{C}_{-\lambda} \to G/B, \,\,
[g,v]\mapsto gB,\]
where $\mathbb{C}_{-\lambda}$ is the one dimensional representation of $B$
corresponding to the character $e^{-\lambda}$ of $B$ and $[g,v]$ denotes the
equivalence class of $(g,v)\in  G\times \mathbb{C}_{-\lambda}$ under the $B$-action:
$b\cdot (g,v)=(gb^{-1}, b\cdot v).$

Let $\{\alpha_1,\ldots,\alpha_{\ell}\}\subset \mathfrak{t}^{*}$ be
the set of simple roots and
$\{\alpha_1^{\vee},\ldots,\alpha^{\vee}_{\ell}\}\subset
\mathfrak{t}$ the set of simple coroots, where $\mathfrak{t}=\Lie
T$. Let $\rho\in \mathfrak{t}^{*}$ be any integral weight satisfying
$$
\rho(\alpha^{\vee}_{i})=1,\quad\text{for all}\quad 1\leq i\leq \ell.
$$

When $G$ is a finite dimensional semisimple group, $\rho$ is unique,
but for a general Kac-Moody group $G$, it may not be unique.

For any $v\leq w \in W$, consider the \emph{Richardson variety}
$$
X^{v}_{w}:=X^{v}\cap X_{w},
$$
and its boundary
$$
\partial X^{v}_{w}:= ((\partial X^{v})\cap X_{w})\cup (X^{v}\cap
\partial X_{w})
$$
both endowed with reduced subvariety structure, where $\partial
X_{w}:=X_{w}\backslash (BwB/B)$ and $\partial X^{v}:=X^{v}\backslash
(B^{-}{v}B/B)$.

Writing $\partial X^{v}_{w}=\cup_{i}X_{i}$ as the union of its
irreducible components, the line bundle
$\mathcal{L}(2\rho)_{|X^{v}_{w}}$ can be written as a (Cartier)
divisor (for justification, see $\S$ \ref{sec3.5} -- the proof of
Theorem \ref{thm2.2}):
\begin{equation}
\mathcal{L}(2\rho)_{|X^{v}_{w}}=\cO_{X^v_w}\left( \sum_{i}b_{i}X_{i}\right),
\quad b_{i}\in
\mathbb{N}:=\{1,2,3, \dots \} .\label{eq4}
\end{equation}

Now, take a positive integer $N$ such that $N>b_{i}$ for all $i$,
and consider the $\mathbb{Q}$-divisor on $X^{v}_{w}$:
\begin{equation}
\Delta =\sum_{i}\left(1-\frac{b_{i}}{N}\right)X_{i}.\label{eq5}
\end{equation}

The following theorem is the main result of the paper.

\setcounter{theorem}{1}
\begin{theorem}\label{thm2.2}
For any $v\leq w\in W$, the pair $(X^{v}_{w},\Delta)$ defined above is KLT.
\end{theorem}

In fact, we will show in Lemma \ref{lem3.4} that $\cO_{X^v_w}(-N(K_{X^v_w} +
\Delta))
\cong \mathcal{L}(2\rho)_{|X^{v}_{w}}$ is ample, which proves that $(X^v_w, \Delta)$ is \emph{log Fano}.

We postpone the proof of this theorem until the next two sections.
But we derive the following consequence proved earlier in the finite case
(i.e., in the case when $G$ is a finite dimensional semisimple group) by
Brion-Lakshmibai
(see \cite[Proposition 1]{BrionLakshmibaiAGeometricApproachToStandardMonomialTheory}).

\begin{corollary}\label{coro2.3}
For a dominant integral weight $\lambda$, and any $v\leq w$,
$$
H^{i}(X^{v}_{w},\mathcal{L}(\lambda)_{|X^{v}_{w}})=0,\quad\text{for
  all}\quad i>0.
$$
\end{corollary}

\begin{proof}
By (the subsequent) Lemma \ref{lem3.4}, $N(K_{X^{v}_{w}}+\Delta)$ is a
Cartier divisor corresponding to the line bundle
$\mathcal{L}(-2\rho)_{|X^{v}_{w}}$. Since $\lambda$ is a dominant
weight, the $\mathbb{Q}$-Cartier divisor $D$ is nef and big, where
$ND$ is the Cartier divisor corresponding to the ample line bundle
$\mathcal{L}(N\lambda+2\rho)_{|X^{v}_{w}}$. Thus, the divisor
$K_{X^{v}_{w}}+\Delta+D$ is Cartier and corresponds to the line bundle
$\mathcal{L}(\lambda)_{|X^{v}_{w}}$. Hence, the corollary follows
from the Logarithmic Kawamata-Viehweg vanishing theorem which we state below
(cf. \cite[Theorem 7.26]{DebarreHigherDimensionalAG} or \cite[Theorem 2.70]{KollarMori}).
\end{proof}

\begin{theorem}\label{thm2.4}
Let $(X,\Delta)$ be a KLT pair and let $D$ be a nef and big
$\mathbb{Q}$-Cartier $\mathbb{Q}$-divisor on $X$ such that
$K_{X}+\Delta+D$ is a Cartier divisor. Then, we have
$$
H^{i}(X,K_{X}+\Delta+D)=0,\quad\text{for all}\quad i>0.
$$
\end{theorem}

\section{Proof of Theorem \ref{thm2.2}: Finite case}\label{sec3}

In this section, except where otherwise noted, we assume that $G$ is a finite dimensional
semisimple simply-connected group. We refer to this as the {\em
finite case}.

We first give a proof of Theorem \ref{thm2.2} in the finite case. In
this case, the proof is much simpler than the general (symmetrizable)
Kac-Moody case proved in the next section. Unlike the general case,
the proof in the finite case given below does not require any use of
characteristic $p>0$ methods.

Before we come to the proof of the theorem, we need some
preliminaries on Bott-Samelson-Demazure-Hansen (for short BSDH)
desingularization of Schubert varieties.

\subsection{BSDH desingularization}\label{sec3.1}

For any $w\in W$, pick a reduced decomposition as a product of simple
reflections:
$$
w=s_{i_1}\ldots s_{i_{n}}
$$
and let $m_{\mathfrak{w}}:Z_{\mathfrak{w}}\to X_{w}$ be the BSDH
desingularization (cf. \cite[\S 2.2.1]{BrionKumarFrobeniusSplitting}), where $\mathfrak{w}$ is
the word $(s_{i_{1}},\dots,s_{i_{n}})$. This is a $B$-equivariant
resolution, which is an isomorphism over the cell $C_w:=BwB/B \subset X_w$.

Similarly, there is a $B^{-}$-equivariant resolution
$$
m^{\mathfrak{v}}:Z^{\mathfrak{v}}\to X^{v},
$$
obtained by taking a reduced word
$\hat{\mathfrak{v}}=(s_{j_{1}}, \dots , s_{j_{m}})$ for $w_{0}v$, i.e.,
$w_{0}v=s_{j_{1}} \ldots s_{j_{m}}$ is a reduced decomposition,
where $w_{0}\in W$ is the longest element. Now, set
$$
Z^{\mathfrak{v}}=Z_{\hat{\mathfrak{v}}},
$$
which is canonically a $B$-variety. We define the action of $B^{-}$
on $Z^{\mathfrak{v}}$ by twisting the $B$-action as follows:
$$
b^{-}\odot z = (\dot{w}_{0}b^{-}\dot{w}^{-1}_{0})\cdot
z,\quad\text{for}\,\, b^{-}\in B^{-}\,\,\,\text{and}\,\, z\in
Z^{\mathfrak{v}},
$$
where $\dot{w}_{0}$ is a lift of $w_{0}$ in the normalizer $N(T)$ of
the torus $T$. (Observe that this action does depend upon the choice of the
lift $\dot{w}_{0}$ of $w_{0}$.) Moreover, define the map
$$
 m^{\mathfrak{v}}:Z^{\mathfrak{v}}\to
  X^{v}=\dot{w}_{0}^{-1}X_{w_{0}v}\quad\text{by}\,\,
  m^{\mathfrak{v}}(z)=\dot{w}^{-1}_{0}(m_{\hat{\mathfrak{v}}}(z)),\,\,\,
  \text{for}\,\,
  z\in Z^{\mathfrak{v}}.
$$

Clearly, $m^{\mathfrak{v}}$ is a $B^{-}$-equivariant
desingularization.

\subsection{Desingularization of Richardson varieties}\label{sec3.2}

We recall the construction of a desingularization of Richardson
varieties communicated to us by M. Brion
(also see \cite[Section 1]{BalanStandardMonomialTheoryForRichardson}).
 It is worked out in detail in any characteristic in
 \cite[Appendix]{KnutsonLamSpeyerProjectionsOfRichardson}. We briefly sketch the
 construction in characteristic zero.  Consider the fiber product
morphism
$$
m^{\mathfrak{v}}_{\mathfrak{w}}:=m^{\mathfrak{v}}\times_{X}m_{\mathfrak{w}}:
Z^{\mathfrak{v}}\times_{X}Z_{\mathfrak{w}}\to
X^{v}\times_{X}X_{w}=X^{v}_{w},
$$
which is a smooth desingularization. It is an isomorphism over the
intersection $C^{v}_{w}:=C^{v}\cap C_{w}$ of the Bruhat cells, where
$C^{v}:=B^{-}vB/B\subset G/B$ and (as earlier) $C_{w}:=BwB/B$. Moreover, the
complement of $C^{v}_{w}$ inside $Z^{\mathfrak{v}}_{\mathfrak{w}}$,
considered as a reduced divisor, is a simple normal crossing
divisor. (To prove these assertions, observe that by Kleiman's
transversality theorem \cite[Theorem 10.8, Chap. III]{Hartshorne}, the fiber product
$Z^{\mathfrak{v}}\times_{X}gZ_{\mathfrak{w}}$ is smooth for a
general $g\in G$ and hence for some $g\in B^{-}B$. But since
$Z_{\mathfrak{w}}$ is a $B$-variety and $Z^{\mathfrak{v}}$ is a
$B^{-}$-variety,
$$
Z^{\mathfrak{v}}\times_{X}gZ_{\mathfrak{w}}\simeq
Z^{\mathfrak{v}}\times_{X}Z_{\mathfrak{w}}.
$$
Moreover,
$Z^{\mathfrak{v}}\times_{X}Z_{\mathfrak{w}}$ is irreducible since
each of its irreducible components is of the same dimension equal to
$\ell (w)-\ell (v)$ and the complement of $C^v_w$ in $Z^{\mathfrak{v}}\times_{X}
Z_{\mathfrak{w}}$ is of dimension $< \ell (w)-\ell (v)$.)
\setcounter{theorem}{2}
\begin{lemma}\label{lem3.3}
With the notation as above (still in characteristic zero), for any $v\leq w$, the Richardson variety $X^{v}_{w}$
is irreducible,
normal and Cohen-Macaulay.
\end{lemma}

This is proven in the finite case in
\cite[Lemma 2]{BrionPositivityInGrothendieckGroupOfFlagVars} and
 \cite[Lemma 1]{BrionLakshmibaiAGeometricApproachToStandardMonomialTheory}. The
 same result (with a similar proof as in \cite{BrionPositivityInGrothendieckGroupOfFlagVars})
also holds in the Kac-Moody case (cf.  \cite[Proposition 6.5]{kum12}).  Also see \cite{KnutsonLamSpeyerProjectionsOfRichardson} for some discussion in characteristic $p > 0$.

\begin{lemma}\label{lem3.4} For any symmetrizable Kac-Moody group $G$, and any
$v\leq w \in W$,
the canonical divisor $K_{X^{v}_{w}}$ of $X^{v}_{w}$ is given by:
$$
K_{X^{v}_{w}}=\mathcal{O}_{X^{v}_{w}}[-\partial X^{v}_{w}],
$$
where $\partial X^{v}_{w}$ is considered as a reduced divisor.
\end{lemma}
\begin{proof}
The finite case can be found in
\cite[Theorem 4.2.1(i)]{BrionLecturesOnGeometryOfFlagVarieties}.  The detailed
proof  for
the general symmetrizable Kac-Moody group can be found in
\cite[Lemma 8.5]{kum12}. We give a brief idea here.

Since $X^v_{w}$ is Cohen-Macaulay  by \cite[Proposition 5.6]{kum12} (in particular,
so is $X_w$) and the codimension of $X^v_{w}$ in $X_w$ is $\ell(v)$,
the dualizing sheaf
\begin{equation} \omega_{X^v_{w}}\simeq \ext^{\ell (v)}_{\co_{X_w}}
\bigl(\co_{X^v_w}, \omega_{X_w}\bigr),\label{eq7.3.1}
\end{equation}
(cf. \cite[Theorm 21.15]{E}).
Observing that
$\depth (\omega_{X_w})= \depth (\co_{X_w})$, as $\co_{X_w}$-modules,
(cf. \cite[Theorem 21.8]{E})
\begin{equation}
\ext^{\ell (v)}_{\co_{X_w}}
\bigl(\co_{X^v_w}, \omega_{X_w}\bigr)\simeq \ext^{\ell (v)}_{\co_{\bar{X}}}
\bigl(\co_{X^v}, \co_{\bar{X}}\bigr)\otimes_{\co_{\bar{X}}} \, \omega_{X_w},\label{eq7.3.2}
  \end{equation}
where $\bar{X}$ is the `thick' flag variety.
By \cite [Proposition 2.2]{GrahamKumarOnPositivityInTEquivariantKTheory},
as $T$-equivariant sheaves,
\begin{equation}
\omega_{X_w}\simeq \bc_{-\rho}\otimes \mathcal{L}(-\rho)\otimes
\co_{X_w}(-\partial X_w).\label{eq7.3.3}
\end{equation}
Similarly, by \cite[Theorem 10.4]{kum12} (due to Kashiwara),
\begin{equation}
\omega_{X^v}\simeq \bc_{\rho}\otimes \mathcal{L}(-\rho)\otimes
\co_{X^v}(-\partial X^v).\label{eq7.3.4}
\end{equation}
Similar to the identity \eqref{eq7.3.1}, we also have
\begin{equation} \omega_{X^v}\simeq \ext^{\ell (v)}_{\co_{\bar{X}}}
\bigl(\co_{X^v}, \omega_{\bar{X}}\bigr),\label{eq7.3.5}
\end{equation}
Since $\omega_{\bar{X}} \simeq  \mathcal{L}(-2\rho)$,
combining the isomorphisms \eqref{eq7.3.1} -
\eqref{eq7.3.5}, we get
$$\omega_{X^v_{w}}\simeq \co_{X^v}(-\partial X^v)\otimes_{\co_{\bar{X}}}\,
\co_{X_w}(-\partial X_w).$$
Now, the lemma follows since all the intersections $X^v\cap X_w$,
$(\partial X^v) \cap X_w$, $ X^v \cap \partial X_w$ and $ \partial (X^v) \cap
 \partial X_w$ are proper. In fact, we need the corresponding local Tor vanishing
 result (cf. \cite[Lemma 5.5]{kum12}).
\end{proof}

We are now ready to prove Theorem \ref{thm2.2} in the finite case.  The basic strategy is similar to the proof that toric varieties have KLT singularities (and in fact are log Fano) \cite[Section 11.4]{CoxLittleSchenckToricVarieties}.

\setcounter{subsection}{4}
\subsection{Proof of Theorem \ref{thm2.2} in the finite
  case}\label{sec3.5}

Let us denote $Z^{\mathfrak{v}}\times_{X}Z_{\mathfrak{w}}$ by
$Z^{\mathfrak{v}}_{\mathfrak{w}}$. Consider the desingularization
$$m^{\mathfrak{v}}_{\mathfrak{w}}:Z^{\mathfrak{v}}_{\mathfrak{w}}\to
X^{v}_{w},$$ as in \S\ \ref{sec3.2}.

By \cite[Proposition 2.2.2]{BrionKumarFrobeniusSplitting}, the canonical line bundle of
$Z_{\mathfrak{w}}$ is isomorphic with the line bundle
$$
\mathcal{L}_{\mathfrak{w}}(-\rho)\otimes
\mathcal{O}_{Z_{\mathfrak{w}}}[-\partial Z_{\mathfrak{w}}],
$$
where $\mathcal{L}_{\mathfrak{w}}(-\rho)$ is the pull-back of the
line bundle $\mathcal{L}(-\rho)$ to $Z_{\mathfrak{w}}$ via the
morphism $m_{\mathfrak{w}}$ and $\partial Z_{\mathfrak{w}}$ is the
reduced divisor $Z_{\mathfrak{w}}\backslash C_{w}$.

Similarly, the canonical line bundle of $Z^{\mathfrak{v}}$ is
isomorphic with
$$\mathcal{L}_{\mathfrak{v}}(-\rho)\otimes
\mathcal{O}_{Z^{\mathfrak{v}}}[-\partial Z^{\mathfrak{v}}],$$
 where
$\mathcal{L}_{\mathfrak{v}}(-\rho)$ is the pull-back of the line
bundle $\mathcal{L}(-\rho)$ to $Z^{\mathfrak{v}}$ via
$m^{\mathfrak{v}}$ and $\partial Z^{\mathfrak{v}}$ is the reduced
divisor $Z^{\mathfrak{v}}\backslash C^{v}$.

Thus, by adapting the proof of \cite[Lemma 1]{BrionPositivityInGrothendieckGroupOfFlagVars},
\begin{align}
\omega_{Z^{\mathfrak{v}}_{\mathfrak{w}}} &\cong
  \omega_{Z^{\mathfrak{v}}}\otimes_{\mathcal{O}_{X}}\omega_{Z_{\mathfrak{w}}}
  \otimes_{\mathcal{O}_{X}}(m^{\mathfrak{v}}_{\mathfrak{w}})^{*}\omega^{-1}_{X}\notag\\
&= \mathcal{O}_{Z^{\mathfrak{v}}_{\mathfrak{w}}}[-\partial
    Z^{\mathfrak{v}}_{\mathfrak{w}}],\label{eq6}
\end{align}
where $\partial Z^{\mathfrak{v}}_{\mathfrak{w}}$ is the reduced
divisor
$$
(\partial Z^{\mathfrak{v}}\times_{X}Z_{\mathfrak{w}})\cup
(Z^{\mathfrak{v}}\times_{X}\partial Z_{\mathfrak{w}}).
$$

Consider the
desingularization
$$
m^{\mathfrak{v}}_{\mathfrak{w}}:Z^{\mathfrak{v}}_{\mathfrak{w}}\to
X^{v}_{w}.
$$
Note that $m^{\mathfrak{v}}_{\mathfrak{w}}$ is an isomorphism outside of $\partial
    Z^{\mathfrak{v}}_{\mathfrak{w}}$.

By Lemma \ref{lem3.4},
$$
K_{X^{v}_{w}}=-\sum X_{i},
$$
where we have written $\partial X^{v}_{w}=\cup X_{i}$ as the union of prime divisors.
 Thus, by \eqref{eq5},
\begin{equation}
K_{X_{w}^{v}}+\Delta=-\frac{1}{N}\sum b_{i}X_{i},\label{eq7}
\end{equation}
which is a $\mathbb{Q}$-Cartier divisor by \eqref{eq4}.

We next calculate $\Exc(m^{\mathfrak{v}}_{\mathfrak{w}})$ and the proper
transform $\Delta'$ of $\Delta$ under the desingularization
$m^{\mathfrak{v}}_{\mathfrak{w}}:Z^{\mathfrak{v}}_{\mathfrak{w}}\to
X^{v}_{w}$.

Any irreducible component $X_{i}$ of $\partial X^{v}_{w}$ is of the
form $X^{v'}_{w}$ or $X^{v}_{w'}$ for some $v\to v'$ and $w'\to w$,
where the notation $w'\to w$ means that $\ell(w)=\ell(w')+1$ and
$w=s_{\alpha}w'$, for some reflection $s_{\alpha}$ through a positive
(not necessarily simple) root $\alpha$. Conversely, any $X^{v'}_{w}$
and $X^{v}_{w'}$ (for $v\to v'$ and $w'\to w$) is an irreducible
component of $\partial X^{v}_{w}$. Thus, we have the prime
decomposition
$$
\partial X^{v}_{w}=\left(\cup_{v\to v'}X^{v'}_{w}\right)\cup
\left(\cup_{w'\to w}X^{v}_{w'}\right).
$$

We define $Z_{i}$ as the prime divisor (of
the resolution $Z^{\mathfrak{v}}_{\mathfrak{w}}$)
$Z^{\mathfrak{v}'}\times_{X} Z_{\mathfrak{w}}$ if $X_{i}=X^{v'}_{w}$
or $Z^{\mathfrak{v}}\times_{X}Z_{\mathfrak{w}'}$ if
$X_{i}=X^{v}_{w'}$. Thus, the strict transform of $\Delta$ can be written as
\begin{equation}
\Delta'=\sum_{i}\left(1-\frac{b_{i}}{N}\right)Z_{i}.\label{eq8}
\end{equation}

We now calculate $E=E_{m^{\mathfrak{v}}_{\mathfrak{w}}}(\Delta)$.
By definition (cf. \eqref{eq1}),
\begin{equation}
E=(K_{Z^{\mathfrak{v}}_{w}}+\Delta')-
\frac{1}{N}(m^{\mathfrak{v}}_{\mathfrak{w}})^{*}(N(K_{X^{v}_{w}}+\Delta)).
\label{eq9}
\end{equation}
Consider the prime decomposition of the reduced divisor $\partial
Z^{\mathfrak{v}}_{\mathfrak{w}}$:
\begin{equation}
\partial
Z^{\mathfrak{v}}_{\mathfrak{w}}=\left(\cup_{i}Z_{i}\right)\cup
\left(\cup_{j}Z'_{j}\right),\label{eq10}
\end{equation}
where $Z'_{j}$ are the irreducible components of $\partial
Z^{\mathfrak{v}}_{\mathfrak{w}}$ which are not of the form $Z_{i}$.
The line bundle $\mathcal{L}(2\rho)_{|X^{v}_{w}}$ has a section
vanishing exactly on the set $\partial X^{v}_{w}$. To see this,
consider the Borel-Weil isomorphism
$$
\beta:V(\rho)^{*}\to H^{0}(G/B,\mathcal{L}(\rho)),\quad
\beta(\chi)(gB)=[g,({g}^{-1}\chi)_{|\mathbb{C}v_{+}}],
$$
where $V(\rho)$ is the irreducible $G$-module with highest weight
$\rho$ and $v_{+}\in V(\rho)$ is a highest weight vector. Let $\chi_{v}$
be the unique (up to a scalar multiple) vector of $V(\rho)^{*}$ with
weight $-v \rho$. Now, take the section $\beta(\chi_{v})\cdot
\beta(\chi_{w})$ of the line bundle $\mathcal{L}(2\rho)_{|X^{v}_{w}}$.
Then, it has the zero set precisely equal to $\partial X^{v}_{w}$,
since the zero set $Z(\beta(\chi_{v})_{|X^{v}})$ of
$\beta(\chi_{v})_{|X^{v}}$ is given by
\begin{align*}
Z(\beta(\chi_{v})_{|X^{v}}) &= \{gB\in
X^{v}:\chi_{v}(gv_{+})=0\}\\[3pt]
&= \bigcup_{v'>v}B^{-}v'B/B\\[3pt]
&= \partial X^{v}.
\end{align*}
We fix $H = \sum_i b_i X_i$ to be the divisor corresponding to
the section  $\beta(\chi_{v})\cdot
\beta(\chi_{w})$ of the line bundle $\mathcal{L}(2\rho)_{|X^{v}_{w}}$
as in \eqref{eq4}.  Observe  that the coefficients of $m^{\mathfrak{v}}_{\mathfrak{w}}{}^{*} H = \sum_i b_i Z_i + \sum_j d_j Z_j'$ are all strictly positive integers.

Thus, by combining the identities \eqref{eq4},
\eqref{eq6}--\eqref{eq10}, we get
\begin{align*}
E &=
-\sum_{i}\frac{b_{i}}{N}{Z}_{i}-\sum_{j}Z'_{j}+\frac{1}{N}
m^{\mathfrak{v}}_{\mathfrak{w}}{}^{*}(H)\\[3pt]
&=
-\sum_{i}\frac{b_{i}}{N}Z_{i}-\sum_{j}Z'_{j}+\frac{1}{N}\sum_{i}b_{i}Z_{i}+\frac{1}{N}\sum
d_{j}Z'_{j},\\[3pt]
%&\qquad\text{since~ } Z_{i}\to X_{i}\text{~ is birational}\\[3pt]
&= \sum_{j}\left(\frac{d_j}{N}-1\right)Z'_{j},
\end{align*}
for some $d_{j}\in \mathbb{N}$ (since the zero set of a certain
section of $\mathcal{L}(2\rho)_{|X^{v}_{w}}$ is precisely equal to
$\partial X^{v}_{w}$ and all the $Z_j'$ lie over $\partial X^v_w$). Thus, the coefficient $e_{j}$ of $Z'_{j}$ in
$E$ satisfies $-1<e_{j}$.

Finally, observe that $\Exc(m^{\mathfrak{v}}_{\mathfrak{w}})+\Delta'$
is a $\Q$-divisor with simple normal crossings since $\Supp\,
(\Exc(m^{\mathfrak{v}}_{\mathfrak{w}})+\Delta')\subset \partial
Z^{\mathfrak{v}}_{\mathfrak{w}}$ and the latter is a simple normal
crossing divisor since
$$\partial
Z^{\mathfrak{v}}_{\mathfrak{w}} =
Z^{\mathfrak{v}}_{\mathfrak{w}} \setminus C^v_w,$$
(cf. $\S$\ref{sec3.2}).

This completes the proof of Theorem \ref{thm2.2} in the finite case.\hfill$\Box$

\begin{remark}
The above proof crucially uses the explicit BSDH  type resolution of the
Richardson varieties $X^v_w$ given in $\S$\ref{sec3.2}. This resolution is
available in the finite case,
but we are not aware of such an explicit resolution in the Kac-Moody case. This is
the main reason that we need to handle the general Kac-Moody case differently.

\end{remark}

\section{Proof of Theorem \ref{thm2.2} in the Kac-Moody
  case}\label{sec4}

Our proof of Theorem \ref{thm2.2} in the general Kac-Moody case is
more involved. It requires the use of characteristic $p$ methods; in
particular, the Frobenius splitting.

For the construction of the flag variety $X=G/B$, Schubert subvarieties $X_w$,
opposite Schubert subvarieties $X^v$ (and thus the Richardson varieties $X^v_w$)
associated to any Kac-Moody group $G$ over an algebraically closed field $k$, we refer to \cite{tits1}, \cite{tits2}, \cite{tits3},
\cite{mat88} and \cite{MathieuConstructionDunGroupeDeKacMoody}.

Let $k$ be an algebraically closed field of
characteristic  $p>0$. Let $Y: Y_0 \subset Y_1\subset \dots $ be an ind-variety
over $k$ and let $\mathcal{O}_{Y}$ be its structure sheaf
(cf.  \cite[Definition 4.1.1]{KumarKacMoodyGroupsBook}). The
{\it absolute Frobenius morphism}
$$
F_Y:Y\longrightarrow Y
$$
is the identity on the underlying space of $Y$, and the $p$-th power
map on the structure sheaf $\Ocal_Y$. Consider the
$\mathcal{O}_{Y}$-linear Frobenius map
\[F^\#: \mathcal{O}_{Y} \to F_*\mathcal{O}_{Y}, f\mapsto f^p.\]

Identical to the definition of Frobenius split varieties, we have the following
 definition for ind-varieties.

\begin{definition}
\label{indsplit} An ind-variety $Y$ is called {\it Frobenius split} (or just {\it split}) if
the $\Ocal_Y$-linear map $F^{\#}$
splits, i.e., there exists an $\Ocal_Y$-linear map
$$
\varphi:F_*\Ocal_Y \longrightarrow \Ocal_Y
$$
such that the composition $\varphi\circ F^{\#}$ is the identity of
$\Ocal_Y$. Any such $\varphi$ is called a {\it splitting}.

A closed ind-subvariety $Z$ of $Y$ is {\it compatibly split} under the
 splitting $\varphi$ if
$$
\varphi(F_*\Ical_Z)\subseteq \Ical_Z,
$$
where $\Ical_Z \subset \Ocal_Y$ is the ideal sheaf of $Z$.

Clearly, a splitting of $Y$ is equivalent to a sequence of splittings
$\varphi_n$ of $Y_n$ such that $\varphi_n$ compatibly splits  $Y_{n-1}$ inducing
the splitting $\varphi_{n-1}$ on $Y_{n-1}$.
\end{definition}

Let $B$ be the standard Borel subgroup of any Kac-Moody group $G$
over an algebraically closed field $k$ of
characteristic $p>0$ and $T\subset B$ the standard maximal torus. For any real root $\beta$,
let $U_\beta$ be the corresponding root subgroup.  Then, there exists an
algebraic group isomorphism $\veps_\beta: \bg_a \to U_{\beta}$ satisfying
  \[
t \veps_\beta (z)t^{-1} = \veps_{\beta}(\beta(t)z),
  \]
for $z\in \bg_a$ and $t\in T$. For any $B$-locally finite algebraic representation
 $V$ of $B$, $v\in V$ and $z\in \bg_a$,
  \[
{\veps}_\beta (z)\, v = \sum_{m\geq 0} z^m \bigl( e^{(m)}_{\beta} \cdot
v\bigr)  ,
  \]
where $e^{(m)}_{\beta}$ denotes the $m$-th divided power of the root vector
$e_{\beta}$, which is, by definition, the derivative of ${\veps}_\beta $ at $0$.

Now, we come to the definition of $B$-canonical splittings for ind-varieties
(cf. \cite[Section 4.1]{BrionKumarFrobeniusSplitting} for more details in the finite case).

\begin{definition}
\label{canonical} Let $Y$ be a $B$-ind-variety, i.e., $B$ acts on the ind-variety
$Y$ via
ind-variety isomorphisms.
   Let $\End_F(Y) := \Hom (F_*\co_Y,\co_Y)$ be
the additive group of all the $\co_Y$-module maps $F_*\co_Y \to \co_Y$.
Recall that   $F_*\co_Y$ can canonically be identified with $\co_Y$
as a sheaf of abelian groups on $Y$; however,
the $\co_Y$-module structure is given by $f\odot g := f^pg$, for $f,g\in
\co_Y$.  Since $Y$ is a  $B$-ind-variety, $B$ acts on
$\End_F(Y)$ by
  \[
(b*\psi)s = b(\psi(b^{-1}s)),
\;\text{ for }\; b\in B, \psi\in\End_F(Y)\;\text{and} \; s\in F_*\co_Y,
  \]
where the action of $B$ on $F_*\co_Y$ is defined to be the standard action
of $B$ on $\co_Y$ under the identification $F_*\co_Y = \co_Y$ (as sheaves
of abelian groups). We define the
$k$-linear structure on $\End_F(Y)$ by
 \[
(z*\psi )s = \psi (zs) = z^{1/p} \psi(s),\]
 for $z\in k, \psi\in \End_F(Y)$ and $s\in\co_Y$.

A splitting $\phi\in \End_F(Y)$  is called a $B$-{\it
canonical splitting} if $\phi$ satisfies the following conditions:
\begin{enumerate}
\item[(a)] $\phi$ is $T$-invariant, i.e.,
  \[
t*\phi = \phi, \quad\text{ for all }t\in T.
  \]
\item[(b)] For any simple root $\al_i$, $1\leq i\leq \ell$, there exist
$\phi_{i,j}\in\End_F(X)$, $0\leq j\leq p-1$, such that
  \beqn
\veps_{\al_i}(z)*\phi = \sum^{p-1}_{j=0} z^j*\phi_{i,j},\quad\text{ for
all }z\in \bg_a .
  \eeqn
\end{enumerate}

  \end{definition}
The definition of $B^-$-canonical is of course parallel.

Before we come to the proof of Theorem \ref{thm2.2} for the Kac-Moody
case, we need the following results.

The following result in the symmetrizable Kac-Moody case is due to O. Mathieu
(unpublished). (For a proof in the finite case, see
\cite[Theorem 2.3.2]{BrionKumarFrobeniusSplitting}.)
Since Mathieu's proof is unpublished, we briefly give an outline of his proof
contained in \cite{mat12}.

\begin{proposition}\label{prop4.1}
Consider the Richardson variety $X^{v}_{w}(k)$ (for
any $v\leq w$) over an algebraically closed field $k$ of
characteristic $p>0$. Then, $X^{v}_{w}(k)$ is Frobenius split compatibly
splitting its boundary $\partial X^{v}_{w}$.
\end{proposition}

\begin{proof}

\medskip
\noindent
{\bf Assertion I:~ The full flag variety {\boldmath$X=X(k)$} admits a
  {\boldmath$B$}-canonical splitting.}
\smallskip

For any $w\in W$ and any reduced decomposition $\mathfrak{w}$ of
$w$, consider the BSDH desingularization $Z_{\mathfrak{w}}=Z_{\mathfrak{w}}(k)$
of the Schubert variety $X_w$ and the  section $\sigma \in
H^{0}(Z_{\mathfrak{w}},\mathcal{O}_{Z_{\mathfrak{w}}}[\partial Z_{\mathfrak{w}}])$
with the associated divisor of zeroes $(\sigma)_{0}=\partial
Z_{\mathfrak{w}}$. Clearly, such a section is unique (up to a nonzero
scalar multiple). Take the unique, up to nonzero multiple, nonzero section $\theta\in
H^{0}(Z_{\mathfrak{w}},\mathcal{L}_{\mathfrak{w}}(\rho))$ of weight $-\rho$.
(Such a section exists since $H^{0}(Z_{\mathfrak{w}},\mathcal{L}_{\mathfrak{w}}
(\rho)) \to H^0(\{1\}, \mathcal{L}_{\mathfrak{w}}(\rho)_{|\{1\}})$ is surjective
by \cite[Theorem 3.1.4]{BrionKumarFrobeniusSplitting}, where
$1:=[1, \dots, 1]\in Z_{\mathfrak{w}}$ and
$[1, \dots, 1]$ denotes the $B^{\ell(w)}$-orbit of $(1, \dots, 1)$ as in
\cite[Definition 2.2.1]{BrionKumarFrobeniusSplitting}. Moreover, such a section is
 unique
up to a scalar multiple since $H^{0}(Z_{\mathfrak{w}},\mathcal{L}_{\mathfrak{w}}
(\rho))^*
\simeq H^{0}(X_w,\mathcal{L}(\rho)_{|X_w})^*\hookrightarrow V(\rho)$.)  By the
above, the section $\theta$ does
not vanish at the base point $1$.
Thus, by
\cite[Proposition 1.3.11 and Proposition 2.2.2]{BrionKumarFrobeniusSplitting},
$(\sigma\theta)^{p-1}$ provides a splitting
$\hat{\sigma}_{\mathfrak{w}}$ of $Z_{\mathfrak{w}}$ compatibly
splitting $\partial Z_{\mathfrak{w}}$. Since the Schubert variety
$X_{w}$ is normal, the splitting $\hat{\sigma}_{\mathfrak{w}}$
descends to give a splitting $\hat{\sigma}_{w}$ of $X_{w}$
compatibly splitting all the Schubert subvarieties of $X_w$.

Now, the splitting $\hat{\sigma}_{\mathfrak{w}}$ is $B$-canonical and
it is the unique $B$-canonical splitting of $Z_{\mathfrak{w}}$
(cf. \cite[Exercise 4.1.E.2]{BrionKumarFrobeniusSplitting}; even though this exercise is for
finite dimensional $G$, the same proof works for the Kac-Moody
case). We claim that the induced splittings $\hat{\sigma}_{w}$ of
$X_{w}$ are compatible to give a splitting of
$X=\cup_{w}X_{w}$. Take $v$, $w\in W$ and choose  $u\in W$
with $v\leq u$ and $w\leq u$. Choose a reduced word $\mathfrak{u}$ of
$u$. Then there is a reduced subword $\mathfrak{v}$
(resp. $\mathfrak{w}$) of $\mathfrak{u}$ corresponding to $v$ (resp. $w$). The
$B$-canonical splitting $\hat{\sigma}_{\mathfrak{u}}$ of
$Z_{\mathfrak{u}}$ (by the uniqueness of the $B$-canonical splittings
of $Z_{\mathfrak{v}}$) restricts to the $B$-canonical splitting
$\hat{\sigma}_{\mathfrak{v}}$ of $Z_{\mathfrak{v}}$ (and
$\hat{\sigma}_{\mathfrak{w}}$ of $Z_{\mathfrak{w}}$). In particular,
the splitting $\hat{\sigma}_{u}$ of $X_{u}$ restricts to the splitting
$\hat{\sigma}_{v}$ of $X_{v}$ (and $\hat{\sigma}_{w}$ of
$X_{w}$). This proves the assertion that the splittings
$\hat{\sigma}_{u}$ of $X_{u}$ are compatible to give a $B$-canonical
splitting $\hat{\sigma}$ of $X$. By the same proof as that of
\cite[Proposition 4.1.10]{BrionKumarFrobeniusSplitting}, we obtain that the $B$-canonical
splitting $\hat{\sigma}$ of $X$ is automatically $B^{-}$-canonical.

\medskip
\noindent
{\bf Assertion II:~ The splitting {\boldmath$\hat{\sigma}$} of
  {\boldmath$X$} canonically  splits the {\boldmath$T$}-fixed points
  of {\boldmath$X$}.}
\smallskip

Take a $T$-fixed point $\dot{w}B\in X$ (for some $w\in W$) and
consider the Schubert variety $X_{w}$. Then, $\dot{w}B\in X_{w}$ has an
affine open neighborhood $U_{w}\simeq U_{w}\cdot \dot{w}B/B$,
where $U_{w}$ is the unipotent subgroup of $G$ with Lie algebra
$\displaystyle{\mathop{\oplus \mathfrak{g}_{\alpha}}_{\alpha\in
    R^{+}\cap w R^{-}}}$. In particular, the ring $k[U_{w}\dot{w}B/B]$
of regular functions, as a $T$-module, has weights lying in the cone
$\sum\limits_{\alpha\in R^{-}\cap wR^{+}}\mathbb{Z}_{+}\alpha$ and the
$T$-invariants $k[U_{w}\dot{w}B/B]^{T}$ in $k[U_{w}\dot{w}B/B]$ are
only the constant functions. Since any $B$-canonical splitting is
$T$-equivariant by definition, it takes the $\lambda$-eigenspace
$
k[U_{w}\dot{w}B/B]_{\lambda}$ to
$k[U_{w}\dot{w}B/B]_{\lambda/p}
$
(cf. \cite[\S4.1.4]{BrionKumarFrobeniusSplitting}). This shows that
the ideal of
$\{\dot{w}\}$ in $k[U_{w}\dot{w}B/B]$ is stable under
$\hat{\sigma}$. Thus $\{\dot{w}\}$ is compatibly split under
$\hat{\sigma}$.

\medskip
\noindent
{\bf Assertion III:~ {\boldmath$X^{w}$} is compatibly split under
  {\boldmath$\hat{\sigma}$}}.
\smallskip

Since $\hat{\sigma}$ is $B^{-}$-canonical, by
\cite[Proposition 4.1.8]{BrionKumarFrobeniusSplitting}, for any
closed ind-subvariety $Y$ of $X$ which is compatibly
split under $\hat{\sigma}$, the $B^{-}$-orbit closure
$\overline{B^{-}Y}\subset X$ is also compatibly split. In
particular, the opposite Schubert variety
$X^{w}:=\overline{B^{-}\dot{w}B/B}$ is compatibly split.

Thus, we get that the Richardson varieties $X^{v}_{w}$ (for $v\leq w$)
are compatibly split under the splitting $\hat{\sigma}$ of $X$.  Since the
 boundary $\partial X^{v}_{w}$ is a union of other Richardson varieties,
 the boundary also is compatibly split. This
proves the proposition.
\end{proof}

We need the following general result. First we recall a definition.

\begin{definition}
\label{def.SharplyFPure}
Suppose that $X$ is a normal variety over an algebraically closed field of
characteristic $p>0$ and $D$ is an effective $\mathbb{Q}$-divisor
 on $X$.  The pair $(X,D)$ is called {\em sharply $F$-pure} if, for every point $x \in X$, there exists an integer $e \geq 1$ such that
$e$-iterated Frobenius map $$\mathcal{O}_{X, x}\to
F^{e}_{*}\big(\mathcal{O}_{X,x}(\lceil (p^{e}-1)D \rceil) \big)$$ admits an
$\mathcal{O}_{X,x}$-module splitting.  In fact, if there exists a splitting for
one $e > 0$, by composing maps, we obtain a Frobenius splitting for all sufficiently divisible $e > 0$.
\end{definition}

Note that by definition, if $\cO_X \to F^e_* \cO_X$ is split relative to a divisor
$D$, then the pair $(X, {1 \over p^e - 1} D)$ is sharply $F$-pure.

Note that being sharply $F$-pure is a purely local condition, unlike being $F$-split.

\begin{proposition}\label{prop4.3}
Let $X$ be an irreducible normal variety over a field of characteristic $p > 0$ and $D=\sum\limits_{i}D_{i}$
a reduced divisor in $X$.  % such that $X\backslash \Supp D$ is smooth.
Assume further that $X$ is Frobenius split compatibly splitting $\Supp
D$. Then, the pair $(X,D)$ is sharply $F$-pure.
\end{proposition}

\begin{proof}
Note that we have a global splitting of $\cO_X(-D) \to F_* \cO_X(-D)$.  Twisting
both sides by $D$ and applying the projection formula gives us a global splitting
of $\cO_X \to F_* \cO_X( (p-1)D)$.  We may localize this at any stalk and
take $e = 1$.
%Use \cite[Proposition 2.2(5)]{HW}; also see \cite[Theorem 3.11(d)]{S}.
\end{proof}

By Lemma
\ref{lem3.3}, the Richardson varieties $X^{v}_{w}$ are normal in characterstic $0$;
in particular, they are normal in characterstics $p\gg 0$. Thus,
combining Propositions \ref{prop4.1} and \ref{prop4.3}, we get the following.

\begin{corollary}\label{coro4.4}
With the notation as above, for any $v\leq w$, $(X^{v}_{w},\partial X^{v}_{w})$ is
sharply $F$-pure in characterstics $p\gg 0$.

\end{corollary}

We also recall the following from \cite[Theorem 3.7]{HaraWatanabeFRegFPure}. It should be mentioned that even though in loc. cit. the result is proved in the local situation, the same proof works for projective varieties.  We sketch a proof for the convenience of the reader.

\begin{theorem}\label{thm4.5}
Let $X$ be an irreducible normal variety over a field of characteristic $0$
and let $D$ be an effective $\mathbb{Q}$-divisor on $X$ such that
$K_{X}+D$ is $\mathbb{Q}$-Cartier and such that $\lceil D \rceil$ is reduced
and effective (i.e.,  the coefficients of $D$ are in $[0,1]$). If the reduction $(X_{p},D_{p})$
mod $p$ of $(X,D)$ is sharply $F$-pure for infinitely many primes $p$, then
$(X,D)$ is log canonical.
\end{theorem}
\begin{proof} (Sketch)
Fix a log resolution $\pi : \tld{X} \to X$ of $(X, D)$ and write
\[
E_{\pi}(D) - D' = K_{\tld X} - {1 \over n} \big( \pi^{*}(nK_{X}+nD)\big)
\]
for a choice of $K_{\tld X}$ agreeing with $K_X$ wherever $\pi$ is an
isomorphism as in (\ref{eq1}).  We need to show that the coefficients of
$E_{\pi}(D)$ are $\geq -1$.  We reduce the entire setup to some characteristic
$p \gg 0$ where $(X_p, D_p)$ is sharply $F$-pure (for a discussion of this
process, see \cite{HaraWatanabeFRegFPure} or see
 \cite[Chapter 1.6]{BrionKumarFrobeniusSplitting} in the
 special case  when the varieties are defined over $\mathbb{Z}$).

Fix $x \in X_p$.
We have a Frobenius splitting $\phi$:
\[
\xymatrix{
\cO_{X_p,x} \ar@{^{(}->}[r] & F^e_* \cO_{X_p,x} \ar@/_2pc/[rr]_{\phi} \ar@{^{(}->}[r] & F^e_* \cO_{X_p,x}(\lceil (p^e -1)D_{p,x}
\rceil) \ar[r] & \cO_{X_p,x}
},
\]
for some $e \geq 1$.  This splitting $\phi \in \Hom(F^e_* \cO_{X_p,x},
\cO_{X_p,x})$ corresponds to a divisor $B_x \geq \lceil (p^e - 1)D_{p,x} \rceil$
on $\Spec \cO_{X_p,x}$,
 which is linearly equivalent to $(1-p^e)K_{X_p,x}$ as in
 \cite[\S1.3]{BrionKumarFrobeniusSplitting}.  Set $\Delta = {1 \over p^e - 1} B_x$.  Observe that $(p^e -1)(K_{X_p,x} + \Delta)$ is
 linearly equivalent to $0$; and thus $K_{X_p,x} + \Delta$ is $\bQ$-Cartier.
 Since
 $$\Delta = {1 \over p^e - 1} B_x \geq {1 \over p^e - 1} \big\lceil (p^e - 1)
 D_{p,x} \big\rceil \geq D_{p,x},$$
 we know
\[
E_{\pi_p}(D_{p,x}) \geq E_{\pi_p}\left(\Delta\right).
\]
Therefore, it is sufficient to prove that the coefficients of
$E_{\pi_p}\left(\Delta\right)$ are $\geq -1$.  Note that it is
possible that $\pi_p$ is not a log resolution for $\Delta$, but this will not matter
for us.

We can factor the splitting $\phi$ as follows (we leave this verification to the reader):
\[
\xymatrix{
F^e_* \cO_{X_p,x} \ar@/_2pc/[rrr]_{\phi} \ar@{^{(}->}[r] & F^e_* \cO_{X_p,x}(\lceil (p^e - 1) D_{p,x} \rceil) \ar@{^{(}->}[r] & F^e_* \cO_{X_p,x}((p^e - 1)\Delta) \ar[r]^-{\psi} & \cO_{X_p,x}
}.
\]

Let $C$ be any prime exceptional divisor of $\pi_p : {\tld X}_p \to
X_p$
with generic point $\eta$ and let $\cO_{{\tld X}_p, \eta}$ be the associated
valuation ring.  Let $a \in \bQ$  be the coefficient of $C$ in $E_{\pi_p}\left(\Delta\right)$.  There are two cases:
\begin{itemize}
\item[(i)]  $a > 0$
\item[(ii)]  $a \leq 0$.
\end{itemize}
Since we are trying to prove that $a \geq -1$, if we are in case (i), we are
already done.  Therefore, we may assume that $a \leq 0$.
By tensoring $\phi$ with the fraction field $K(X_p) = K(\tld X_p)$, we obtain a map $\phi_{K(X_p)} : F^e_* K(X_p) \to K(X_p)$.

\begin{claim}
By restricting $\phi_{K(X_p)}$ to the stalk $F^e_* \cO_{\tld{X}_p, \eta}$, we
 obtain a map $\phi_{\eta}$ which factors as:
\[
\phi_{\eta} : F^e_* \cO_{\tld{X}_p, \eta} \hookrightarrow F^e_* \cO_{\tld{X}_p, \eta}\Big(-a(p^e - 1) C\Big) \to \cO_{\tld{X}_p, \eta}.
\]
\end{claim}
\begin{proof}[Proof of the claim]
Indeed, similar arguments are used to prove Grauert-Riemenschneider vanishing for
Frobenius split varieties
\cite{MehtaVanDerKallenGRVanishing},
\cite[Theorem 1.3.14]{BrionKumarFrobeniusSplitting}.  We briefly sketch the
idea of the proof.

We identify $\psi$ with a section $s \in \cO_{X_p,x}((1-p^e)(K_{X_p,x} + \Delta))
\cong \Hom_{\cO_{X_p,x}}(F^e_* \cO_{X_p,x}((p^e -1)\Delta ), \cO_{X_p,x})$.  Recall
that $\pi^*(1-p^e)(K_{X_p,x} + \Delta) = (1-p^e)K_{\tld{X}_p, \eta} -
(1-p^e)E_{\pi_p}(\Delta) + (1-p^e)\Delta'$, where $\Delta'$ is the strict transform of $\Delta$.  Thus we can pull $s$ back to a section
\[
\begin{array}{rcl}
t & := & \pi^* s \in \cO_{\tld X_p, \eta}(\pi^* (1 - p^e)(K_{X_p,x} + \Delta))\\
 & = & \cO_{\tld X_p, \eta}\left((1 - p^e)K_{\tld X_p, \eta} - (1-p^e)E_{\pi_p}
 \left(\Delta\right) +(1-p^e)\Delta'\right)\\
 & = & \cO_{\tld X_p, \eta}\left( (1-p^e)K_{\tld X_p,\eta} - (1-p^e)aC \right)\\
 & \cong & \Hom_{\cO_{\tld X_p, \eta}}\left(F^e_* \cO_{\tld X_p, \eta}
 (-a(p^e - 1) C), \cO_{\tld X_p, \eta}\right).
\end{array}
\]
It is not hard to see that the homomorphism $\psi_{\eta} : F^e_* \cO_{\tld X_p, \eta}(-a(p^e - 1) C) \to \cO_{\tld X_p, \eta}$ corresponding to $t$ can be chosen to agree with $\psi$ on the fraction field $K(X) = K(\tld X)$, we leave this verification to the reader.  It follows that $\phi_{\eta}$ is the composition $F^e_* \cO_{\tld X_p, \eta} \hookrightarrow F^e_* \cO_{\tld X_p, \eta}(-a(p^e - 1) C) \xrightarrow{\psi_{\eta}} \cO_{\tld X_p, \eta}$.
This concludes the proof of the claim.
\end{proof}

Now we complete the proof of Theorem\ref{thm4.5}.
Note that $\phi_{\eta}$ is a splitting because $\phi$ was a splitting and both the
maps agree on the field of fractions.  Therefore,
 $0 \leq -a(p^e - 1) \leq p^e -1$ since the splitting along a divisor can not
 vanish to order greater than $p^e - 1$.  Dividing by $(1 - p^e)$ proves that $a \geq -1$ as desired.
\end{proof}

We also recall the following.

\begin{lemma}\label{lem4.6}
Let $X$ be an irreducible normal projective variety over $\mathbb{C}$
and let $D=\sum a_{i}D_{i}$ be an effective $\mathbb{Q}$-divisor on
$X$ such that $X\backslash \Supp D$ is smooth and $(X,D)$ is log
canonical. Now, consider a $\mathbb{Q}$-divisor $\Delta=\sum
c_{i}D_{i}$ with $c_{i}\in [0,1)$ and $c_{i}<a_{i}$ for all $i$, such
that
  $K_{X}+\Delta$ is $\mathbb{Q}$-Cartier. Then, $(X,\Delta)$ is KLT.
\end{lemma}
\begin{proof}
We may choose a resolution of singularities $\pi : \tld X \to X$ which is a log
resolution for  $(X, D)$ (and  hence also for $(X, \Delta)$) by
\cite{HironakaResolution}.  Furthermore, we may assume that $\pi$ is an
isomorphism over $X\setminus \Supp(D) \subseteq X\setminus \Supp(\Delta)$.
 Therefore, we see that
\[
E_{\pi}(\Delta) = K_{\tld X} - {1 \over n} \pi^*\big(n(K_X + \Delta)\big) + \Delta'
\geq K_{\tld X} - {1 \over n} \pi^*\big(n(K_X + D)\big) + D' = E_{\pi}(D)
\]
with strict inequality in every nonzero coefficient.  Since every coefficient of $E_{\pi}(D)$
is $\geq -1$, we are done.
\end{proof}

We now come to the proof of Theorem \ref{thm2.2} in the Kac-Moody case.

\begin{proof}[Proof of Theorem \ref{thm2.2}]
\label{sec4.7}
We begin by reducing our entire setup to characteristic $p \gg 0$.

By Corollary \ref{coro4.4}, $(X^{v}_{w}(k),\partial X^{v}_{w}(k))$
is sharply $F$-pure for any algebraically closed field $k$ of  characteristic
$p\gg 0$. Moreover, $K_{X_{w}^{v}}+\partial X^{v}_{w}=0$
 by Lemma \ref{lem3.4}; in particular, it is Cartier.

Hence, returning now to characteristic zero, by Theorem \ref{thm4.5}, $(X^{v}_{w},\partial X^{v}_{w})$ is
log canonical. Furthermore, the $\mathbb{Q}$-divisor
$\Delta=\sum\limits_{i}(1-\frac{b_{i}}{N})X_{i}$, where $\partial
X^{v}_{w}=\sum X_{i}$, clearly satisfies all the assumptions of
Lemma \ref{lem4.6}. Thus, $(X^{v}_{w},\Delta)$ is KLT, proving
Theorem \ref{thm2.2} in the Kac-Moody case as well.%\hfill$\Box$
\end{proof}

\begin{remark}
One can define KLT singularities in positive characteristic too by considering
all valuations on all normal birational models.  In particular, a similar argument
 shows that any normal Richardson variety is  KLT  in characteristic $p > 0$ as
 well. (It is expected that, for any symmetrizable Kac-Moody group, all the
 Richardson varieties $X^v_w$ are normal in any characteristic.)
\end{remark}

%\bibliographystyle{skalpha}%\ok{--}{--}
%\bibliography{CommonBib}

\def\cprime{$'$} \def\cprime{$'$}
  \def\cfudot#1{\ifmmode\setbox7\hbox{$\accent"5E#1$}\else
  \setbox7\hbox{\accent"5E#1}\penalty 10000\relax\fi\raise 1\ht7
  \hbox{\raise.1ex\hbox to 1\wd7{\hss.\hss}}\penalty 10000 \hskip-1\wd7\penalty
  10000\box7}
\providecommand{\bysame}{\leavevmode\hbox to3em{\hrulefill}\thinspace}
\providecommand{\MR}{\relax\ifhmode\unskip\space\fi MR}
\providecommand{\MRhref}[2]{%
  \href{http://www.ams.org/mathscinet-getitem?mr=#1}{#2}
}
\providecommand{\href}[2]{#2}

%Addresses:

%\noindent

%\vskip1ex

%\noindent
%K.S.: Department of Mathematics, Penn State University,
%University Park, PA 16802, USA  (email: schwede$@$math.psu.edu)

\end{document}